\documentclass[reqno, 11pt]{amsart}

\usepackage{amsmath,amsfonts,amssymb,amsthm}
\usepackage{xfrac}

\renewcommand{\geq}{\geqslant}

\newtheorem{theorem}{Theorem}[section]
\newtheorem{lemma}[theorem]{Lemma}
\newtheorem{corollary}[theorem]{Corollary}
\newtheorem{proposition}[theorem]{Proposition}

\theoremstyle{definition}

\newtheorem{example}[theorem]{Example}

\newtheorem{remark}[theorem]{Remark}
\newtheorem{remarks}[theorem]{Remarks}

\newcommand{\fix}{\operatorname{Fix}}
\newcommand{\aut}{\operatorname{Aut}}
\newcommand{\stab}{\operatorname{Stab}}
\newcommand{\bigo}{\operatorname{O}}
\newcommand{\littleo}{\operatorname{o}}

\renewcommand{\geq}{\geqslant}

\newcommand{\nfix}{\mathsf{F}}
\newcommand{\norb}{\mathsf{O}}

\newcommand{\entropy}{\mathsf{h}}
\newcommand{\ugr}{\mathsf{g}}

\title[A dynamical zeta function]{A dynamical zeta function for group actions}

\author{Richard Miles}
\email{richard.miles@math.uu.se}
\thanks{This work was initiated with the support of the London Mathematical Society
Scheme~IV grant number~$41352$, which the author gratefully acknowledges.}
\begin{document}

\begin{abstract}
This article introduces and investigates the basic features of a dynamical zeta function for group actions, motivated by the classical dynamical zeta function of a single transformation. A product formula for the dynamical zeta function is established that highlights a crucial link between this function and the zeta function of the acting group. A variety of examples are explored, with a particular focus on full shifts and closely related variants. Amongst the examples, it is shown that there are infinitely many non-isomorphic virtually cyclic groups for which the full shift has a rational zeta function. In contrast, it is shown that when the acting group has Hirsch length at least $2$, a dynamical zeta function with a natural boundary is more typical. The relevance of the dynamical zeta function in questions of orbit growth is also considered.
\end{abstract}
\maketitle

\section{Introduction}

The zeta function of a dynamical system is a fundamental invariant that has warranted considerable attention since the definitive work of Artin and Mazur~\cite{MR0176482}. Ruelle~\cite{MR1920859} provides an introduction to various guises of this function, and a comprehensive survey by Sharp in the context of periodic orbits of hyperbolic flows may be found in~\cite{MR2035655}. For a discrete time dynamical system given by a transformation $T:X\rightarrow X$ of a state space~$X$, the dynamical zeta function is defined formally by
\[
\zeta_T(z)=\exp\sum_{n\geqslant 1}\frac{\nfix_T(n)}{n}z^n,
\]
where $\nfix_T(n)=|\{x\in X:T^nx=x\}|$. As well as being an important invariant in its own right, $\zeta_T$ is often related to other dynamical invariants, especially if the dynamical system~$(X,T)$ is framed in a topological or measurable context~\cite{MR0271401},~\cite{MR0176482},~\cite{MR0228014}. For example, if~$T$ is a hyperbolic toral automorphism, then the radius of convergence of~$\zeta_T$ is $\exp(-\entropy)$, where $\entropy$ is the topological entropy. Natural applications involving~$\zeta_T$ may be found in studies of orbit growth~\cite{MR710244},~\cite{MR727704},~\cite{MR1139101}. 

In this article, we introduce and investigate the basic features of a dynamical zeta function for group actions, motivated by the classical dynamical zeta function described above. This programme can be traced back to the pioneering work of Lind~\cite{MR1411232}, who introduced a dynamical zeta function for a~$\mathbb{Z}^d$-action by homeomorphisms of a compact metric space. A compatible definition has been introduced for actions of the infinite dihedral group $\mathbf{D}_\infty$, by Kim, Lee and Park~\cite{MR1978372}, and the definition given here generalizes both these, and the definition used by Artin and Mazur. Notably, Kimoto~\cite{MR2013092} has also considered alternative definitions of zeta functions associated with finite group actions. 

Let~$G$ be a finitely generated group and let~$\mathcal{L}_G$ denote the collection of finite index subgroups of~$G$. We tacitly assume throughout that $G$ is infinite. For any~$L\in\mathcal{L}_G$, write~$[L]=[G:L]$ when the supergroup is clear. Since~$G$ is finitely generated, for all $n\geqslant 1$, the collection of subgroups of index $n$ in $G$ is finite. Denote this collection by $\mathcal{L}_G(n)$ and write $a_G(n)=|\mathcal{L}_G(n)|$. Suppose that~$\alpha$ is a~$G$-action on a set~$X$ such that for all~$L\in\mathcal{L}_G$ the cardinality~$\nfix_\alpha(L)$ of the set of \emph{$L$-periodic points}
\begin{equation}\label{periodic_point_set}
\fix_\alpha(L)=\{x\in X:gx=x\mbox{ for all }g\in L\}
\end{equation}
is finite. The \emph{dynamical zeta function} of~$\alpha$ is defined formally by
\begin{equation}\label{first_action_zeta_equation}
\zeta_\alpha(z)=\exp\sum_{L\in\mathcal{L}_G}\frac{\nfix_\alpha(L)}{[L]}z^{[L]}.
\end{equation}
Note that for an invertible transformation~$T$, the classical dynamical zeta function $\zeta_T$ agrees with the definition above, with $\alpha$ being the $\mathbb{Z}$-action generated by~$T$. 

A central focus in the investigation here is the relationship between~$\zeta_\alpha$ and the zeta function of the acting group~\cite{MR943928} (see also~\cite{MR1978431} and~\cite{MR2371185}), defined by the formal Dirichlet series 
\begin{equation*}\label{first_group_zeta_equation}
\zeta_G(z)=\sum_{n\geqslant 1}\frac{a_G(n)}{n^{z}}.
\end{equation*}
In particular, a key tool will be the Dirichlet series
\begin{equation}\label{exponent_dirichlet_series}
\Delta_{L}(z)=\frac{\zeta_{L}(z+1)}{\zeta(z+1)}=\sum_{n\geqslant 1}\frac{b_{L}(n)}{n^z}, 
\end{equation}
where $L\in\mathcal{L}_G$, and $\zeta=\zeta_\mathbb{Z}$ denotes the Riemann zeta function. Note that~$\Delta_L$ has rational coefficients for all $L\in\mathcal{L}_G$. The relationship between these Dirichlet series and~$\zeta_\alpha$ arises via the development of product formulae for~$\zeta_\alpha$. The classical dynamical zeta function has the well known product formula
\begin{equation}\label{classical_case_product_formula}
\zeta_\alpha(z)=\prod_{Y}(1-z^{|Y|})^{-1},
\end{equation}
where $Y$ runs through all finite orbits of the $\mathbb{Z}$-action~$\alpha$, and this is often regarded as the dynamical analogue of the Euler product  for the Riemann zeta function. In analogy with the Ihara zeta function for a finite connected graph (for example, see~\cite{MR2768284}), one could begin with a product formula as the definition of a dynamical zeta function for a group action. However,~(\ref{classical_case_product_formula}) does not correspond to~(\ref{first_action_zeta_equation}) in general nor, in particular, with the existing studies~\cite{MR1411232} and~\cite{MR1978372}, except in the case~$d=1$ in~\cite{MR1411232}. 

When deriving a product formula formally equivalent to~(\ref{first_action_zeta_equation}), one needs to take into account the various isomorphism classes of stabilizers in $G$, and this introduces a significant level of complexity. 
Before stating the main product formula for~$\zeta_\alpha$, first note that for any finite orbit~$Y$, all stabilizers of elements of~$Y$ lie in the same isomorphism class, for which we may choose an appropriate representative in~$L(Y)\in\mathcal{L}_G$.  Note also that our hypotheses ensure that there are only finitely many orbits of a given cardinality (see Lemma~\ref{moebius_lemma}), so the product given below is properly defined.

\begin{theorem}\label{explicit_product_theorem}
Let~$G$ be a finitely generated group and let~$\alpha$ be an action of~$G$
on a set~$X$ such that~$\nfix_\alpha(L)<\infty$, for all~$L\in\mathcal{L}_G$. Then the dynamical zeta function~$\zeta_\alpha(z)$ is given by the product
\begin{equation}\label{now_is_the_time}
\zeta_\alpha(z)=\prod_{Y}\prod_{n\geqslant 1}(1-z^{|Y|n})^{-b_{L(Y)}(n)},
\end{equation}
where~$Y$ runs over all finite orbits in~$X$,~$L(Y)$ is the stabilizer of any element of~$Y$ and the exponents~$b_{L(Y)}(n)\in\mathbb{Q}$, $n\geqslant 1$, are the coefficients of the Dirichlet series~$\Delta_{L(Y)}$ given by~(\ref{exponent_dirichlet_series}). Hence, if $\Delta_{L}$ has integer coefficients for all~$L\in\mathcal{L}_G$, then the formal power series for~$\zeta_\alpha$ has integer coefficients.
\end{theorem}

The concluding sentence of the theorem highlights part of the motivation for obtaining a product formula for~$\zeta_\alpha$. The presence of  integer coefficients in the formal power series for~$\zeta_\alpha$ has has proved to be important when studying the analytic behaviour of both the classical dynamical zeta function~\cite{MR3217030}, random dynamical zeta functions~\cite{MR1925634} and Lind's dynamical function~\cite{MR1411232},~\cite{MR3336617},~\cite{miles_ward_zero_dimensional}, since the celebrated P{\'o}lya-Carlson theorem~\cite{MR1544479}, \cite{MR1512473} may then be brought to bear.  It is also worth noting that product formulae for Lind's dynamical zeta function feature prominently in certain studies of higher-dimensional shifts of finite type~\cite{MR3025123}. 

The second main result concerns full shifts of groups having polynomial subgroup growth (PSG groups) and Hirsch length at least $2$ (for a definition of Hirsch length, see Section~\ref{finite_orbit_section}). This is obtained using a simplified form of~(\ref{now_is_the_time}) that we deduce for full shifts (Theorem~\ref{full_shifts_theorem}) coupled with conclusions that we draw (Proposition~\ref{natural_boundary_proposition}) from Shalev's classification of groups with bounded subgroup growth~\cite{MR1475693} and the P{\'o}lya-Carlson theorem. 

\begin{theorem}\label{full_shift_natural_boundary_theorem}
Let $G$ be a polycyclic-by-finite PSG group such that the Dirichlet series~$\Delta_G$ has integer coefficients. If $G$ has Hirsch length at least~$2$, then for the full shift $\sigma(G)$ of $G$ on $\mathcal{A}^G$, where $\mathcal{A}$ is a finite alphabet, the dynamical zeta function~$\zeta_{\sigma(G)}$ has a natural boundary at the circle of convergence $|z|=|\mathcal{A}|^{-1}$.
\end{theorem}

We also include a related result (Corollary~\ref{connected_group_corollary}) for actions generated by automorphisms of compact connected abelian groups, obtained by applying the authors previous work~\cite{MR3336617}, that further motivates Theorem~\ref{explicit_product_theorem}.

The examples explored emphasize the role of~$\zeta_G$ and the importance of explicit formulae for~$\zeta_G$ that have been obtained, as for example by Grunewald, Segal and Smith~\cite{MR943928} and du Sautoy, McDermott and Smith~\cite{MR1710163}. Note that the development of such formulae is now quite advanced~\cite{MR2371185}, and we investigate only a small fraction of the many possibilities available here. Our examples show that when the acting group is virtually cyclic,~$\zeta_\alpha$ may in fact be rational and, in contrast with~\cite{MR1411232} and~\cite{MR1978372}, we find that there are infinitely many non-isomorphic groups for which the full shift has a rational dynamical zeta function. Beyond full shifts, the subgroup conjugacy class structure in~$\mathcal{L}_G$ is seen to play a more significant role and this is illustrated with the help of an action $\alpha$ of the planar symmetry group~$\mathbf{pm}\cong\mathbb{Z}^2\rtimes\mathbf{C}_2$ on a projected shift space. This group was chosen because~$\mathcal{L}_\mathbf{pm}$ has a non-trivial partition into finitely many isomorphism classes, $\Delta_L$ has integer coefficients for all $L\in\mathcal{L}_G$, and~$\mathbf{pm}$ has the advantage of being a fundamental expositional example treated in~\cite{MR1710163}. For this action, we find that~$\zeta_\alpha$ may be calculated with the aid of shifted factorials. 

The paper is structured as follows. Section~\ref{finite_orbit_section} assembles some preliminaries concerning finite orbits of group actions and also develops some fundamental properties of the zeta function, including the consequences of the authors earlier paper~\cite{MR3336617} and our deductions based on Shalev's work and the P{\'o}lya-Carlson theorem. Also in this section, in  light of earlier studies by the author and Ward~\cite{MR2465676},~\cite{ETS:9315807}, by using M\"obius inversion on the incidence algebra of $\mathcal{L}_G$, we highlight how the dynamical zeta function may find further applications in studies of orbit growth. In Section~\ref{product_formulae_section}, we develop product formulae for $\zeta_\alpha$, giving the proof of Theorem~\ref{explicit_product_theorem} in two natural stages. Section~\ref{full_shift_section} contains the more detailed investigation of full shifts and related examples. Finally, we conclude with a list of open problems.

\section{Finite orbits and fundamental properties of $\zeta_\alpha$}\label{finite_orbit_section} 

We begin by recalling some important invariants of finitely generated groups. Let $G$ be such a group. When $G$ is polycyclic-by-finite, there is a finite index subgroup $G_m\leqslant G$ with subnormal series
\begin{equation}\label{subnormal_series}
\{1\}=G_0\lhd G_1\lhd G_2\lhd\dots\lhd G_{m},
\end{equation}
where $G_{k}/G_{k-1}$ is infinite cyclic for all $1\leqslant k\leqslant m$. The \emph{Hirsch length} of~$G$ is equal to~$m$ and this is an invariant of~$G$. Let~$\mathcal{L}_G$ denote the collection of finite index subgroups of~$G$, and let~$\mathcal{L}_G(n)$ denote the collection of subgroups of index~$n$. If~$a_G(n)=|\mathcal{L}_G(n)|$ grows at most polynomially in~$n$, then~$G$ is said to have \emph{polynomial subgroup growth}. For example, if~$G$ is nilpotent then~$G$ has polynomial
subgroup growth. When $G$ has polynomial subgroup growth, the Dirichlet series~(\ref{first_group_zeta_equation}) is frequently used as a generating function for the sequence $(a_G(n))$. For further details concerning groups and their subgroup growth properties, see~\cite{MR1978431}.

There is a natural partial order on $\mathcal{L}_G$ given by subgroup inclusion and, 
since~$a_G(n)<\infty$ for all~$n\geqslant 1$, the
poset~$\mathcal{L}_G$ is locally finite, that is each interval 
\[
[L, K]=\{M\in\mathcal{L}_G:L\leqslant M\leqslant K\}
\]
is a finite set. This means that one may define a
corresponding M\"obius function~$\mu$ on the set of intervals by 
\[
\mu(L,L)=1\mbox{ for all }L\in\mathcal{L}_G
\]
and
\[
\mu(L,K)=-\sum_{L\leqslant M<K}\mu(L, M)\mbox{ for all }L<K\mbox{ in }\mathcal{L}_G.
\]
Note that the classical M\"obius function on the group of integers
coincides with the M\"obius function on the set of intervals
in the lattice of finite index subgroups of~$\mathbb{Z}$
on writing~$\mu(n)=\mu(n\mathbb{Z},\mathbb{Z})$.
For further background concerning this combinatorial
framework, see Stanley's book~\cite{MR1442260}.

Let~$\mathcal{O}(\mathcal{L}_G)$ denote the set of orbits that arise from the
natural conjugation action of~$G$ on~$\mathcal{L}_G$
given by~$g\cdot L=gLg^{-1}$,
where~$g\in G$ and~$L\in\mathcal{L}_G$.
This action is clearly trivial if~$G$ is abelian. Note also that conjugate subgroups are isomorphic and have the same index. Since there are only finitely many subgroups of any given index in~$G$, all orbits of this action are finite. 
The cardinality~$|\Omega|$ of an orbit~$\Omega\in\mathcal{O}(\mathcal{L}_G)$ containing a subgroup~$L$ is equal to the index of the normalizer
$N_G(L)=\{g\in G: gLg^{-1}=L\}$ of~$L$, as this is also the stabilizer of~$L$ under the action of subgroup conjugation. 

Now suppose~$\alpha$ is a $G$-action on a set~$X$,
with the property that the cardinality $\nfix_\alpha(L)$ of the set of $L$-periodic points $\fix_\alpha(L)$ given by~(\ref{periodic_point_set})
is finite for all~$L\in\mathcal{L}_G$. We observe the following elementary properties.

\begin{lemma}\label{basic_fixed_point_lemma} Let the group~$G$ and the $G$-action~$\alpha$ be as above.
Then
\begin{enumerate}
\item\label{membership_via_stabilizers} $x\in\fix_\alpha(L)\Leftrightarrow L\leqslant\stab(x)$
\item $g\fix_\alpha(L)=\fix_\alpha(gLg^{-1})$
\item\label{stabilizer_formula} $\stab(gx)=g\stab(x)g^{-1}$
\end{enumerate}
for any~$L\in\mathcal{L}_G$,~$x\in X$ and~$g\in G$.
\end{lemma}

By Lemma~\ref{basic_fixed_point_lemma}(\ref{stabilizer_formula}),
the equivalence relation defined by
\[
x\sim y \Longleftrightarrow \stab(x)=\stab(y)
\]
also satisfies, for all~$g\in G$,
\[
x\sim y\Rightarrow gx\sim gy.
\]
This defines a \emph{block system} for the action of~$G$ on~$X$. A finite orbit~$Y\subset X$ is partitioned into \emph{blocks} of the form
\[
B(L)=\{x\in Y:\stab(x)=L\},
\]
for the subgroups~$L\in\mathcal{L}_G$ that occur as stabilizers of elements of~$Y$. So, for any block~$B\subset Y$ and any~$g\in G$, we have either
$gB=B$ or $gB\cap B=\varnothing$.
The blocks comprising~$Y$ have the same cardinality, which is
$|Y|$ divided by the number of distinct subgroups of~$G$ occurring as stabilizers of elements of~$Y$. Therefore, if~$Y$ is an orbit containing an element with stabilizer~$L$, since the action of~$G$ on~$Y$ is transitive, Lemma~\ref{basic_fixed_point_lemma}(\ref{stabilizer_formula}) shows that the number of such distinct subgroups is given by~$|\Omega(L)|$, 
where~$\Omega(L)\in\mathcal{O}(\mathcal{L}_G)$ is the unique orbit containing~$L$.
Clearly~$L\leqslant N_G(L)$ and~$\stab(L)=N_G(L)$. Therefore,~$|\Omega(L)|=[N_G(L)]$ and $[G:N_G(L)][N_G(L):L]=[G:L]$.
In particular,~$|\Omega(L)|$ divides the orbit length~$|Y|=[L]$, where this equality follows from the orbit-stabilizer theorem. Hence, each block forming part of the orbit~$Y$ has cardinality 
\begin{equation}\label{block_size}
[L]/|\Omega(L)|=[L]/[N_G(L)].
\end{equation} 

The following example illustrates how the block system functions for a dynamical system arising from an action of the infinite dihedral group~$\mathbf{D}_\infty$.

\begin{example}\label{block_system_example}
Consider the group $\mathbf{D}_\infty=\langle a,b: b^2=1, ab=ba^{-1}\rangle$.
For each odd index~$n$, there are exactly~$n$ distinct subgroups of~$\mathbf{D}_\infty$ of index~$n$, 
\begin{equation}\label{subgroup_list}
\langle a^n, b\rangle,\langle a^n, ab\rangle,\langle a^n, a^2b\rangle,\dots,\langle a^n, a^{n-1}b\rangle, 
\end{equation}
and for each even index~$n$, in addition to the subgroups (\ref{subgroup_list}), there is one additional subgroup~$\langle a^{n/2}\rangle$ of index~$n$. An action~$\alpha$ of $\mathbf{D}_\infty$ on the torus $\mathbb{T}^2=(\mathbb{R}/\mathbb{Z})^2$ is generated by multiplication by the matrices
\[
A=
\begin{pmatrix}
-2 & 3\\
1&-2
\end{pmatrix}
\mbox{ and }
B=
\begin{pmatrix}
7 & -12\\
4&-7
\end{pmatrix}
\]
which satisfy $B^2=1$ and $AB=BA^{-1}$. A method for calculating periodic points for this example is explained in~\cite[Ex.~2.2]{ETS:9315807}. Here, it is shown how the finite orbits of this action intersect~$\fix_\alpha(L)$ for the particular subgroup $L=\langle a^6, b\rangle$, according to the block system just described. The subgroup~$L$ has  conjugates~$\langle a^6, a^2 b\rangle$ and $\langle a^6, a^4 b\rangle$ and supergroups $\langle a^3, b\rangle$, $\langle a^2, b\rangle$ and~$\langle a, b\rangle$. Amongst the supergoups, only $\langle a^3, b\rangle$ has distinct conjugates, namely $\langle a^3, ab\rangle$ and $\langle a^3, a^2b\rangle$. All the subgroups just mentioned occur as stabilizers of points in the finite orbits of~$\alpha$ that intersect $\fix_\alpha(L)$. These~$7$ orbits $Y_1,\dots,Y_7$ are illustrated by the rows in Table~\ref{block_system_diagram}. Stabilizers of points are shown above them, and the set~$\fix_\alpha(L)$ comprises the~$12$ points in the final column of the table. Individual cells in the table containing the periodic points correspond to blocks in the block system for~$\alpha$.
\begin{center}
\begin{table}[h]
\caption{Orbits intersecting $\fix_\alpha(\langle a^6, b\rangle)$ in Example~\ref{block_system_example}.\label{block_system_diagram}}
\begingroup
\renewcommand{\arraystretch}{1.15}
\begin{tabular}{c|c|c|c|}
\cline{2-4}
Stabilizer&$\langle a^6, a^4b\rangle$&$\langle a^6, a^2b\rangle$&$\langle a^6, b\rangle$
\\ \cline{2-4}
$Y_7$ 
&
$\begin{pmatrix}\sfrac{2}{30}\\ \sfrac{29}{30}\end{pmatrix}$
$\begin{pmatrix}\sfrac{23}{30}\\ \sfrac{26}{30}\end{pmatrix}$
&
$\begin{pmatrix}\sfrac{29}{30}\\ 0\end{pmatrix}$
$\begin{pmatrix}\sfrac{26}{30}\\ \sfrac{15}{30}\end{pmatrix}$
&
$\begin{pmatrix}\sfrac{2}{30}\\ \sfrac{1}{30}\end{pmatrix}$
$\begin{pmatrix}\sfrac{23}{30}\\ \sfrac{4}{30}\end{pmatrix}$
\\ \cline{2-4} 
$Y_6$ 
&
$\begin{pmatrix}\sfrac{4}{30}\\ \sfrac{28}{30}\end{pmatrix}$
$\begin{pmatrix}\sfrac{16}{30}\\ \sfrac{22}{30}\end{pmatrix}$
&
$\begin{pmatrix}\sfrac{28}{30}\\ 0\end{pmatrix}$
$\begin{pmatrix}\sfrac{22}{30}\\ 0\end{pmatrix}$
&
$\begin{pmatrix}\sfrac{4}{30}\\ \sfrac{2}{30}\end{pmatrix}$
$\begin{pmatrix}\sfrac{16}{30}\\ \sfrac{8}{30}\end{pmatrix}$
\\ \cline{2-4} 
$Y_5$
&
$\begin{pmatrix}\sfrac{6}{30}\\ \sfrac{27}{30}\end{pmatrix}$
$\begin{pmatrix}\sfrac{9}{30}\\ \sfrac{18}{30}\end{pmatrix}$
&
$\begin{pmatrix}\sfrac{27}{30}\\ 0\end{pmatrix}$
$\begin{pmatrix}\sfrac{18}{30}\\ \sfrac{15}{30}\end{pmatrix}$
&
$\begin{pmatrix}\sfrac{6}{30}\\ \sfrac{3}{30}\end{pmatrix}$
$\begin{pmatrix}\sfrac{9}{30}\\ \sfrac{12}{30}\end{pmatrix}$
\\ \cline{2-4} 
$Y_4$
&
$\begin{pmatrix}\sfrac{8}{30}\\ \sfrac{26}{30}\end{pmatrix}$
$\begin{pmatrix}\sfrac{2}{30}\\ \sfrac{14}{30}\end{pmatrix}$
&
$\begin{pmatrix}\sfrac{26}{30}\\ 0\end{pmatrix}$
$\begin{pmatrix}\sfrac{14}{30}\\ 0\end{pmatrix}$
&
$\begin{pmatrix}\sfrac{8}{30}\\ \sfrac{4}{30}\end{pmatrix}$
$\begin{pmatrix}\sfrac{2}{30}\\ \sfrac{16}{30}\end{pmatrix}$\\
\cline{2-4}
Stabilizer&$\langle a^3, ab\rangle$&$\langle a^3, a^2b\rangle$&$\langle a^3, b\rangle$
\\ \cline{2-4} 
$Y_3$ 
&
$\begin{pmatrix}\sfrac{20}{30}\\ \sfrac{20}{30}\end{pmatrix}$
&
$\begin{pmatrix}\sfrac{20}{30}\\ 0\end{pmatrix}$
&
$\begin{pmatrix}\sfrac{20}{30}\\ \sfrac{10}{30}\end{pmatrix}$
\\ \cline{2-4} 
\multicolumn{3}{r|}{Stabilizer}&$\langle a^2, b\rangle$
\\ \cline{4-4} 
\multicolumn{3}{r|}{$Y_2$}
&
$\begin{pmatrix}0\\ \sfrac{15}{30}\end{pmatrix}$
$\begin{pmatrix}\sfrac{15}{30}\\ 0\end{pmatrix}$
\\  \cline{4-4}
\multicolumn{3}{r|}{Stabilizer}&$\langle a, b\rangle$
\\ \cline{4-4} 
\multicolumn{3}{r|}{$Y_1$}
&
$\begin{pmatrix}0\\ 0\end{pmatrix}$
\\  \cline{4-4}
\end{tabular}
\endgroup
\end{table}
\end{center}
\end{example}

Given any~$L\in\mathcal{L}_G$, write~$\norb_\alpha(L)$ for the number of orbits of the action~$\alpha$ on~$X$ that contain at least one element~$x$ with~$\stab(x)=L$. The following lemma connects~$\norb_\alpha(L)$ and~$\nfix_\alpha(L)$. 
\begin{lemma}\label{moebius_lemma}
Let~$G$ be a finitely generated group and let~$\alpha$ be a~$G$-action
 a set~$X$ such that for all~$L\in\mathcal{L}_G$,~$\nfix_\alpha(L)<\infty$.
Then, for any~$L\in\mathcal{L}_G$,
\begin{enumerate}
\item\label{subgroup_fixed_point_formula}
$\displaystyle\nfix_\alpha(L)=\sum_{K\geq L}\frac{[K]}{[N_G(K)]}\norb_\alpha(K)$,
\item\label{inversion_formula}
$\displaystyle\norb_\alpha(L) = \frac{[N_G(L)]}{[L]}\sum_{K\geq L}\mu(L,K)\nfix_\alpha(K)$.
\end{enumerate}
\end{lemma}

\begin{proof}
Orbits of the $G$-action $\alpha$  on $X$ intersect~$\fix_\alpha(L)$ in blocks, and~(\ref{subgroup_fixed_point_formula}) follows immediately from Lemma~\ref{basic_fixed_point_lemma}(\ref{membership_via_stabilizers}) and~(\ref{block_size}). Consequently,~(\ref{inversion_formula}) is obtained using the dual form of M\"obius inversion~\cite[Prop.~3.7.2]{MR1442260}.
\end{proof}

\begin{remark}
Lemma~\ref{moebius_lemma} corrects an earlier oversight by the author and Ward in their work on orbit counting, specifically in~\cite[Eq.~(2)]{MR2465676} and~\cite[Eq.~(2)]{ETS:9315807}, where~$\norb_\alpha(L)$ is defined as it may be more simply when~$G$ is abelian, and the factor $[N_G(L)]$ is omitted. However, the main results in~\cite{MR2465676} and~\cite{ETS:9315807} remain unaffected because this missing factor is cancelled in the principal orbit counting formula, derived here with the appropriate correction for the non-abelian case in Proposition~\ref{principal_orbit_counting_proposition}.
\end{remark}

Recall the definition of~$\zeta_\alpha$ given by~(\ref{first_action_zeta_equation}). It is also sometimes useful to write 
\[
\zeta_\alpha(z)=
\exp\sum_{n\geqslant 1}\frac{1}{n}\sum_{L\in\mathcal{L}_G(n)}\nfix_\alpha(L)z^n.
\]
Note that the formal power series for $\zeta_\alpha$ is thus obtained by composition of generating functions in the usual way. 

If~$\alpha$ and~$\beta$ are actions of a finitely generated group~$G$ on sets~$X$ and~$Y$ respectively and there is a bijection~$\phi:X\rightarrow Y$ such that for all~$g\in G$ and all~$x\in X$,
\[
\phi(gx)=g\phi(x),
\]
then, for all~$L\in\mathcal{L}_G$,~$\nfix_\beta(L)=\nfix_\alpha(L)$. Hence,~$\zeta_\alpha$ and~$\zeta_\beta$ are identical. Thus, the dynamical zeta function is an invariant of bijective conjugacy.

If $\alpha$ is an action of a finitely generated PSG group $G$ on a set $X$, we define
\begin{equation}\label{ugr_definition}
\ugr(\alpha)=
\limsup_{[L]\rightarrow\infty}
\frac{1}{[L]}
\log\nfix_\alpha(L)
=\lim_{n\rightarrow\infty}
\sup_{[L]\geqslant n}
\frac{1}{[L]}
\log\nfix_\alpha(L).
\end{equation}
Since~$G$ has polynomial subgroup growth, if~$\nfix_\alpha(L)<\infty$ for all~$L\in\mathcal{L}_G$, then the usual Hadamard formula shows that~$\zeta_\alpha$ has radius of convergence~$\exp(-\ugr(\alpha))$.  

A $G$-action on a compact metric space $(X,\rho)$ is called \emph{expansive} if there exists $r>0$ such that $\sup_{g\in G}\rho(gx,gy)\geqslant r$ for all distinct $x,y\in X$. In the classical setting of a topological dynamical system where~$\alpha$ is generated by a single expansive homeomorphism of a compact metric space, Bowen~\cite{MR0399413} shows that there exist constants $C_1,C_2>0$ such that for sufficiently large~$n$,
\begin{equation}\label{bowens_inequalities}
C_1\exp(\entropy(\alpha)n)\leqslant\nfix_\alpha(\langle n \rangle)\leqslant C_2\exp(\entropy(\alpha)n),
\end{equation}
where $\entropy(\alpha)$ is the topological entropy.  Hence, $\ugr(\alpha)=\entropy(\alpha)$. For non-expansive $\mathbb{Z}$-actions, periodic point counting estimates so closely tied to entropy should not be expected (for example, see~\cite{MR684244}). Moreover, for more general group actions, even with the assumption of expansiveness, there is little reason to expect a direct relationship between $\ugr(\alpha)$ and~$\entropy(\alpha)$. For instance, Lind~\cite{MR1411232} gives an example of an expansive $\mathbb{Z}^2$-action~$\alpha$ by continuous automorphisms of a compact abelian group such that $\ugr(\alpha)=\log 4$ and $\entropy(\alpha)=\log 3$. Another particular case for which $\ugr(\alpha)=\entropy(\alpha)$, is that of an expansive action~$\alpha$ of $\mathbf{D}_\infty$ by continuous automorphisms of a compact abelian group, considered by the author in~\cite{ETS:9315807}. This is achieved by obtaining 
bounds analogous to~(\ref{bowens_inequalities}) for all subgroups of index $n$. However, in Section~\ref{full_shift_section}, it will be seen that no such analogous result is possible for expansive actions of all other virtually cyclic groups (that is, those that are neither isomorphic to $\mathbb{Z}$ nor~$\mathbf{D}_\infty$) when $\entropy(\alpha)>0$.

In~\cite{MR3336617}, it is shown that a mixing $\mathbb{Z}^2$-action $\alpha$ by automorphisms of compact connected finite dimensional abelian group~$X$ satisfies $\ugr(\alpha)=0$. Using this result as a basis, we see that a similar result holds much more generally. For an algebraic characterization of mixing, see~\cite{MR1345152}. In the context of the theorem, it ensures $\nfix_\alpha(L)<\infty$ for all~$L\in\mathcal{L}_G$.
 
\begin{theorem}\label{slow_growth_rate_theorem}
Let $\alpha$ be a mixing action of a polycyclic-by-finite PSG group~$G$ by automorphisms of a compact connected finite dimensional abelian group~$X$. If $G$ has Hirsch length at least~$2$, then $\ugr(\alpha)=0$.
\end{theorem}

\begin{proof}
Firstly, since $\alpha$ is mixing, any non-trivial subaction is also mixing. Consider a subnormal series of $G$ of the form~(\ref{subnormal_series}), so that $G_2\cong \mathbb{Z}\rtimes\mathbb{Z}$. Since $\aut(\mathbb{Z})\cong\mathbf{C}_2$, a standard argument shows that this semi-direct product contains a subgroup isomorphic to~$\mathbb{Z}^2$, so $G$ contains a subgroup~$H$ isomorphic to~$\mathbb{Z}^2$, and the subaction $\beta$ of $H$ on~$X$ is also mixing. Furthermore, for any~$L\in\mathcal{L}_G$, $[H:H\cap L]\leqslant [G:L]$ and 
\[
\nfix_\alpha(L)\leqslant \nfix_\alpha(H\cap L)=\nfix_\beta(H\cap L).
\]
Therefore,
\[
\frac{1}{[G:L]}
\log\nfix_\alpha(L)
\leqslant
\frac{1}{[H:H\cap L]}
\log\nfix_\beta(H\cap L)
\]
and so $\ugr(\alpha)\leqslant\ugr(\beta)$. Moreover, since~$\beta$ is mixing,~\cite[Th.~1.1]{MR3336617} shows that
\[
\ugr(\beta)=
\limsup_{[H:K]\rightarrow\infty}
\frac{1}{[H:K]}
\log\nfix_\beta(K)
=
0,
\]
as $K$ runs through finite index subgroups of $H$.
\end{proof}

When $\ugr(\alpha)=0$, an application of the P{\'o}lya-Carlson theorem~\cite[Ch.~6, Th.~5.3]{MR2376066} together with Shalev's classification of groups with bounded subgroup growth~\cite{MR1475693} leads to the following results.

\begin{proposition}\label{natural_boundary_proposition}
Let $\alpha$ be an action of a polycyclic-by-finite PSG group~$G$ of Hirsch length at least~$2$ on a set $X$ such that  $\nfix_\alpha(L)=0$ for at most finitely many $L\in\mathcal{L}_G$. If $\ugr(\alpha)=0$ and the formal power series for $\zeta_\alpha$ has integer coefficients, then $\zeta_\alpha$ has the unit circle as a natural boundary.
\end{proposition}

\begin{proof}
Since~$\zeta_\alpha$ has integer coefficients and radius of convergence~1, the P{\'o}lya-Carlson theorem implies that~$\zeta_\alpha$ is a either a rational function or has the unit circle as a natural boundary. 
As in~\cite[Ex.~3.4]{MR1411232}, suppose for a contradiction that~$\zeta_\alpha$ is rational, that is
\[
\zeta_\alpha(z)=
\frac{\lambda\prod_{i=1}^k(1-\xi_iz)}{\prod_{j=1}^\ell(1-\eta_jz)}
\]
for some $\lambda,\xi_1,\dots,\xi_k,\eta_1,\dots,\eta_\ell\in\mathbb{C}$. Elementary calculus then shows $\lambda=1$ and 
\begin{equation}\label{up_in_the_morning}
\sum_{L\in\mathcal{L}_G(n)}
\nfix_\alpha(L)
=
\sum_{j=1}^\ell\eta_j^n-\sum_{i=1}^k\xi_i^n,
\end{equation}
Since $\zeta_\alpha$ is analytic inside the unit disk, $|\eta_j|\leqslant 1$ for $1\leqslant j\leqslant \ell$. Furthermore, as $\zeta_\alpha$ is the exponential of a convergent power series inside the unit disk, we also have $|\xi_i|\leqslant 1$ for $1\leqslant i\leqslant k$. Hence,
\begin{equation}\label{forced_bounded_subgroup_growth}
a_n(G)\leqslant\sum_{L\in\mathcal{L}_G(n)}\nfix_\alpha(L)\leqslant k + \ell,
\end{equation}
where the first inequality holds for all but finitely many $n$ by hypothesis, and the second inequality holds for all $n$, by~(\ref{up_in_the_morning}). However, Shalev~\cite[Th.~1.2]{MR1475693} has shown that $a_n(G)=\bigo(n)$ if and only if $G$ is virtually cyclic. In particular, since~$G$ has Hirsch length at least~$2$ by hypothesis, the bounded subgroup growth purported by~(\ref{forced_bounded_subgroup_growth}) contradicts Shalev's theorem.
\end{proof}

\begin{corollary}\label{connected_group_corollary}
Let $\alpha$ be a mixing action of a polycyclic-by-finite PSG group~$G$ of Hirsch length at least $2$ by automorphisms of a compact connected finite dimensional abelian group~$X$. If the formal power series for $\zeta_\alpha$ has integer coefficients, then $\zeta_\alpha$ has the unit circle as a natural boundary.
\end{corollary}

\begin{proof}
This follows from Proposition~\ref{natural_boundary_proposition} and Theorem~\ref{slow_growth_rate_theorem}.
\end{proof}

We return to the appearance of a natural boundary for the dynamical zeta function for full shifts in Section~\ref{full_shift_section}. A more detailed investigation of dynamical systems for which the zeta function has a natural boundary is beyond the scope of this article, even for actions generated by automorphisms of compact abelian groups. See~\cite{MR1411232} and~\cite{miles_ward_zero_dimensional} for some of the more subtle issues involved in the case $G=\mathbb{Z}^d$, with $d\geqslant 2$. In this setting, Lind conjectures~\cite[Conj.~7.1]{MR1411232} that $\zeta_\alpha$ is meromorphic for $|z|<\exp(-\overline{\ugr}(\alpha))$, where~$\overline{\ugr}(\alpha)$ is the limit obtained in~(\ref{ugr_definition}) by replacing $[L]$ with 
\[
\min\{||\mathbf{n}||:\mathbf{n}\in L\setminus\{\mathbf{0}\}\}. 
\]
Lind also conjectures that the circle~$|z|=\exp(-\overline{\ugr}(\alpha))$ is a natural boundary for the function. Although some progress is made in addressing Lind's conjecture in~\cite{MR3336617} and~\cite{miles_ward_zero_dimensional}, this remains open in the general context of $\mathbb{Z}^d$-actions by automorphisms of compact abelian groups.

To conclude this section, we consider the dynamical zeta function's utility for counting finite orbits. For an action~$\alpha$ of a finitely generated group~$G$, denote the number of orbits of~$\alpha$ of cardinality at most~$N$ by~$\pi_\alpha(N)$. 

\begin{proposition}\label{principal_orbit_counting_proposition}
Let~$G$ be a finitely generated group and let~$\alpha$
be a~$G$-action on a set~$X$ such that for all~$L\in\mathcal{L}_G$,~$\nfix_\alpha(L)<\infty$. Then
\[
\pi_\alpha(N)=
\sum_{[L]\leqslant N}
\frac{1}{[L]}
\sum_{K\geqslant L}\mu(K,L)\nfix_\alpha(K).
\]
\end{proposition}

\begin{proof}
First note that~$x\in X$ lies in orbit of length~$n$ if and only if~$\stab(x)\in\mathcal{L}_G(n)$. For any~$L\in\mathcal{L}_G(n)$ and for any orbit of the action of~$\alpha$ on~$X$ that contains a point with stabilizer~$L$, the stabilizers of all the other elements in this orbit are of the form~$gLg^{-1}$, for some~$g\in G$, by Lemma~\ref{basic_fixed_point_lemma}(\ref{stabilizer_formula}). There are~$[N_G(L)]$ such stabilizers, all of index~$n$, as this is the length of the orbit~$\Omega$ of~$L$ under the action of~$G$ on~$\mathcal{L}_G(n)$ given by conjugation. Furthermore,~$\norb_\alpha(L)$ gives the number of orbits in~$X$ with a stabilizer in~$\Omega$, for any choice of~$L\in\Omega$. Letting~$\Omega$ run over all orbits in~$\mathcal{L}_G(n)$, we may express the number of orbits of length~$n$ for the action~$\alpha$ as 
\[
\sum_{\Omega}\frac{1}{|\Omega|}\sum_{L\in\Omega}\norb_\alpha(L)
=\sum_{\Omega}\sum_{L\in\Omega}\frac{\norb_\alpha(L)}{[N_G(L)]}
=\sum_{L\in\mathcal{L}_G(n)}\frac{\norb_\alpha(L)}{[N_G(L)]}
\]
which, by Lemma~\ref{moebius_lemma}, is equal to 
\[
\sum_{L\in\mathcal{L}_G(n)}
\frac{1}{[L]}
\sum_{K\geqslant L}\mu(L,K)\nfix_\alpha(K).
\]
The formula for~$\pi_\alpha(N)$ now follows.
\end{proof}

Following~\cite{MR1978431}, write $s_G(N)=\sum_{n\leqslant N}a_G(n)$. Define 
\[
f_\alpha(N)=\sup_{[L]\leqslant N}\nfix_\alpha(L)
\mbox{ and }
m_G(N)=\displaystyle\sup_{[L]\leqslant N, K>L}\frac{|\mu(L,K)|}{[L]}.
\]
As a consequence of the orbit counting formula provided by Proposition~\ref{principal_orbit_counting_proposition}, the following link between~$\pi_\alpha(N)$ and the coefficients of $\log\zeta_\alpha$ is available if the quantity $f_\alpha(N/2)/f_\alpha(N)$ decays sufficiently rapidly.  

\begin{corollary}\label{corollary_for_a_main_term}
Let~$G$ be a finitely generated group and let~$\alpha$
be a~$G$-action on a set~$X$ such that for all~$L\in\mathcal{L}_G$,~$\nfix_\alpha(L)<\infty$ and
\begin{equation}\label{decay_condition}
f_\alpha(N/2)m_G(N)s_G(N) N=\littleo(f_\alpha(N)).
\end{equation}
Then
\[
\pi_\alpha(N)\sim\sum_{[L]\leqslant N}\frac{\nfix_\alpha(L)}{[L]}.
\]
\end{corollary}

\begin{proof}
By Proposition~\ref{principal_orbit_counting_proposition}
\[
\pi_\alpha(N)
=
\sum_{[L]\leqslant N}\frac{\nfix_\alpha(L)}{[L]}
+
\underbrace{
\sum_{[L]\leqslant N}
\frac{1}{[L]}
\sum_{K>L}\mu(K,L)\nfix_\alpha(K)}_{E(N)}.
\]
For any $L\in\mathcal{L}_G$ with $[L]\leqslant N$, any proper supergroup $K>L$ satisfies $[K]\leqslant N/2$, so 
\[
\frac{E(N)}{\sum_{[L]\leqslant N}\nfix_\alpha(L)/[L]}
\leqslant
\frac{N}{f_\alpha(N)}
f_\alpha(N/2)m_G(N)s_G(N) 
\]
and the required result follows.
\end{proof}

Corollary~\ref{corollary_for_a_main_term} is somewhat crude, and can be improved for more precise orbit growth estimates, but it is intended to illustrate and generalize a key idea that has featured in earlier studies of orbit growth, for example, in~\cite{MR2465676} and~\cite{ETS:9315807}. For many dynamical systems of natural interest, such as the full shifts considered in Section~\ref{full_shift_section}, $f_\alpha(N/2)/f_\alpha(N)$ decays exponentially, whilst the other terms in~(\ref{decay_condition}), that depend only on $G$ rather than on~$\alpha$, grow subexponentially. For example, if~$G$ is a PSG group, then~$s_G(N)$ grows at most polynomially and, if $G$ is nilpotent, the results of Kratzer and Th{\'e}venaz~\cite{MR761806}, together with an application of Crapo's closure theorem~\cite{MR0245483}, show that $m_G(N)$ is subexponential in~$N$. Notably, these observations together with Corollary~\ref{corollary_for_a_main_term} give an explicit realization of the methods of~\cite{MR2465676} in terms of the dynamical zeta function developed here. 

Note that in the classical case of a $\mathbb{Z}$-action, $s_\mathbb{Z}(N)=N$ and $|\mu(\cdot)|\leqslant 1$, so if such an action is generated by an expansive homeomorphism of a compact metric space, then~(\ref{decay_condition}) is satisfied due to the key estimate of Bowen~(\ref{bowens_inequalities}), and Corollary~\ref{corollary_for_a_main_term} again applies.

\begin{example}\label{virtually_cyclic_toral_example}
Consider the group $\mathbb{Z}\times(\mathbb{Z}/p\mathbb{Z})$, where $p$ is a rational prime. We have the following straightforward description of~$\mathcal{L}_{\mathbb{Z}\times(\mathbb{Z}/p\mathbb{Z})}$. If $p\nmid n$, then there is precisely one subgroup of index~$n$, namely 
\[
L(n,0)=\langle(n,0),(0,1)\rangle
\]
and if $p\mid n$, then there are $p$ subgroups of index~$n$ of the form
\[
L(n,k)=\langle(n,0),(kn/p,1)\rangle,
\]
where $0\leqslant k\leqslant p-1$, plus one further subgroup $L(n)=\langle(n/p,0)\rangle$ of index~$n$. 

Now consider the particular case of an action of $G=\mathbb{Z}\times(\mathbb{Z}/3\mathbb{Z})$. The matrices
\[
A
=
\begin{pmatrix}
1& 2& 1& 0\\
-2&3&0&1\\
1&0&0&0\\
0&1&0&0
\end{pmatrix}
\mbox{ and }
B
=
\begin{pmatrix}
0& -1& 0& 0\\
1&-1&0&0\\
0&0&0&-1\\
0&0&1&-1
\end{pmatrix}
\]
satisfy $AB=BA$ and $B^3=I_4$, and together generate a $G$-action $\alpha$ by automorphisms of the torus $\mathbb{T}^4$. The matrix $A$ has characteristic polynomial $x^4-4x^3+5x^2+4x+1$ with non-real roots $\omega_1,\overline{\omega}_1,\omega_2,\overline{\omega}_2$, for which $|\omega_1|\approx 2.764$, $|\omega_2|\approx 0.362$ and $\omega_1\omega_2=-1$. 

The set of fixed points $\fix_\alpha(B)=\{(\frac{j}{3},\frac{-j}{3},\frac{k}{3},\frac{-k}{3})^\mathsf{T}:j,k=0,1,2\}$ is a closed $A$-invariant subgroup of $\mathbb{T}^4$. Since $\fix_\alpha(L(n,0))\leqslant\fix_\alpha(B)$ for all $n\geqslant 1$, a standard computation shows
\begin{equation}\label{there_is_power_in_the_factory}
\nfix_\alpha(L(n,0))=
\left\{
\begin{array}{ll}
9 & \mbox{ if }8\mid n,\\
1 & \mbox{ if }8\nmid n.
\end{array}
\right.
\end{equation}
So, if~$3\nmid n$, then the number of periodic points for the sole subgroup $L(n,0)$ in $\mathcal{L}_G(n)$ is bounded. On the other hand, if $3\mid n$, in addition to the subgroup $L(n,0)\in\mathcal{L}_G(n)$ for which $\nfix_\alpha(L(n,0))\in\{1, 9\}$,
the $3$ subgroups $L(n),L(n,1),L(n,2)\in\mathcal{L}_G(n)$ have exponentially growing sets of periodic points, which are also closed finite subgroups of~$\mathbb{T}^4$. As is standard for compact abelian group automorphisms, groups of periodic points can be calculated using the intersection of kernels of homomorphisms. Upon writing~$j=n/3$, we see~$\nfix_\alpha(L(n))$ is given by
\begin{eqnarray*}
|\ker(A^j-I_4)|
&=&
|\det(A^j-I_4)|\\
&=&
|\omega_1^j-1||\overline{\omega}_1^j-1||\omega_2^j-1||\overline{\omega}_2^j-1|.
\end{eqnarray*}
Furthermore, $\nfix_\alpha(L(n,2))$ is given by 
\begin{equation}\label{hang_with_me}
|\ker(A^{n}-I_4)\cap\ker(A^jB^2-I_4)|
=|\ker(A^{n}-I_4)\cap\ker(A^j-B)|
\end{equation}
and since $A^n-I_4=(A^j-B)(A^{2j}+BA^{j}+B^{2})$, $\ker(A^{n}-I_4)\subset\ker(A^j-B)$. Therefore, the right hand side of~(\ref{hang_with_me}) is 
\[
|\det(A^j-B)|
=
|\omega_1^j-\overline{\delta}||\overline{\omega}_1^j-\delta||\omega_2^j-\overline{\delta}||\overline{\omega}_2^j-\delta|,
\]
where $\delta=\exp(2\pi i/3)$.
A similar calculation shows $\nfix_\alpha(L(n,1))$ is given by 
\[
|\det(A^j-B^2)|
=
|\omega_1^j-\delta||\overline{\omega}_1^j-\overline{\delta}||\omega_2^j-\delta||\overline{\omega}_2^j-\overline{\delta}|.
\] 
Since $\omega_1\omega_2=-1$, by expanding these products of absolute values, we find that $\nfix_\alpha(L(n))+\nfix_\alpha(L(n,1))+\nfix_\alpha(L(n,2))$ is equal to
\begin{equation}\label{power_in_the_land}
\sum_{k=1}^3|\det(A^j-B^k)|=
3(
\lambda_1^j
+\lambda_2^j
+\lambda_3^j
+\lambda_4^j
+2),
\end{equation}
where $\lambda_1=\omega_1\overline{\omega}_1$, $\lambda_2=\omega_1\overline{\omega}_2$,
$\lambda_3=\omega_2\overline{\omega}_1$
and $\lambda_4=\omega_2\overline{\omega}_2$. Hence, combining~(\ref{there_is_power_in_the_factory}) and~(\ref{power_in_the_land}), it follows that there exist constants $C_1, C_2>0$ such that 
\[
C_1\lambda_1^{\lfloor N/3\rfloor} \leqslant f_\alpha(N)\leqslant C_2\lambda_1^{\lfloor N/3\rfloor}.
\]

By~\cite[Prop.~2.4]{MR761806}, for any given finitely generated abelian group, there is a constant $C_3>0$ such that for any $K,L\in\mathcal{L}_G$ with $L\leqslant K$, 
\[
|\mu(L,K)|=|\mu(K/L,0)|\leqslant C_3^{(\log[L])^2}, 
\]
so $m_G(N)\leqslant\frac{1}{N}C_3^{(\log N)^2}$. In addition, $s_G(N)=3\lfloor N/3\rfloor+N$, and so our combined estimates show that~(\ref{decay_condition}) is satisfied. Hence, Corollary~\ref{corollary_for_a_main_term} implies that $\pi_\alpha(N)\sim\sum_{[L]\leqslant N}\nfix_\alpha(L)/[L]$. Thus,
\[
\pi_\alpha(N)\sim\sum_{n\leqslant N}\frac{1}{n}\sum_{L\in\mathcal{L}_G(n)}\nfix_\alpha(L)
\sim
\sum_{n\leqslant \lfloor N/3\rfloor}\frac{\lambda_1^n}{n}
\sim
\frac{\lambda_1^{\lfloor N/3\rfloor+1}}{(\lambda_1-1)\lfloor N/3\rfloor}.
\]

With the aid of~(\ref{there_is_power_in_the_factory}) and (\ref{power_in_the_land}),~$\zeta_\alpha$ may be calculated as follows. 
\begin{eqnarray*}
\zeta_\alpha(z)
&=&
\exp\left(
\sum_{8\mid n}\frac{9}{n}z^n
+
\sum_{8\nmid n}\frac{1}{n}z^n
+
\sum_{3\mid n}\frac{3}{n}\left(2+\sum_{k=1}^4
\lambda_k^{n/3}\right)z^n
\right)\\
&=&
(1-z)^{-1}
(1-z^8)^{-1}
(1-z^3)^{-2}
\prod_{k=1}^4
(1-\lambda_k z^3)^{-1}.
\end{eqnarray*}
In particular,~$\zeta_\alpha$ is rational.
\end{example}

\begin{remark}
An important development in the history of dynamical zeta functions was the use of Tauberian theorems to derive orbit growth asymptotics, as in the work of Parry and Pollicott~\cite{MR710244}, \cite{MR727704}, for example. This avenue opens when~$\zeta_\alpha$ has an analytic continuation beyond the circle of convergence but often, in quite general situations, one can obtain sufficient information concerning periodic point counts to deduce orbit growth estimates via direct methods, aided instead by Corollary~\ref{corollary_for_a_main_term}, as in Example~\ref{virtually_cyclic_toral_example}. 
\end{remark}

\section{Product formulae}\label{product_formulae_section}

For any finite orbit~$Y\subset X$, Lemma~\ref{basic_fixed_point_lemma}(\ref{stabilizer_formula}) shows that the stabilizers of elements of~$Y$ comprise a single orbit in~$\mathcal{O}(\mathcal{L}_G)$, and since conjugate subgroups are isomorphic, we may choose a representative group~$L(Y)\in\mathcal{L}_G$ for this isomorphism class and uniquely define the \emph{trivial action}~$\tau(Y)$ of~$L(Y)$ on a single point, together with its associated dynamical zeta function~$\zeta_{\tau(Y)}$, which is clearly invariant under the isomorphism just described. This allows us to express our initial product formula for the dynamical zeta function.

\begin{lemma}\label{basic_dynamical_zeta_product_lemma}
Let~$G$ be a finitely generated group and let~$\alpha$ be a~$G$-action on a set~$X$ such that for all~$L\in\mathcal{L}_G$,~$\nfix_\alpha(L)<\infty$.  
Then
\[
\zeta_\alpha(z)=\prod_{Y}\zeta_{\tau(Y)}(z^{|Y|}),
\]
where~$Y$ runs over all finite orbits in~$X$.
\end{lemma}

\begin{proof}
Let~$\mathcal{Y}$ denote the set of finite orbits in~$X$ and for each~$L\in\mathcal{L}_G$, let
$\nfix_{\alpha}(Y, L)=|Y\cap\fix_{\alpha}(L)|$.
Let~$\Omega(Y)\in\mathcal{O}(\mathcal{L}_G)$  comprise the set of stabilizers of elements of~$Y$, so that by Lemma~\ref{basic_fixed_point_lemma}(\ref{membership_via_stabilizers}) and~(\ref{block_size})
\[
\nfix_{\alpha}(Y, L)
=
\frac{|Y|}{|\Omega(Y)|}\sum_{K\in\Omega(Y)}\chi(L,K)
\] 
where~$\chi(L,K)=1$ if~$L\leqslant K$ and~$\chi(L,K)=0$ otherwise.
Therefore, 
\begin{eqnarray*}
\zeta_\alpha(z) 
& = &
\exp\sum_{L\in\mathcal{L}_G}
\sum_{Y\in\mathcal{Y}}
\frac{\nfix_{\alpha}(Y, L)}{[L]}
z^{[L]}
\\
& = &
\prod_{Y\in\mathcal{Y}}
\exp\sum_{L\in\mathcal{L}_G}
\frac{\nfix_{\alpha}(Y, L)}{[L]}
z^{[L]}
\\
& = &
\prod_{Y\in\mathcal{Y}}
\exp
\frac{|Y|}{|\Omega(Y)|}
\sum_{L\in\mathcal{L}_G}
\sum_{K\in\Omega(Y)}\frac{\chi(L,K)}{[L]}
z^{[L]}.
\end{eqnarray*}
Reversing the order of summation over~$L$ and~$K$ gives
\[
\zeta_\alpha(z) 
=
\prod_{Y\in\mathcal{Y}}
\exp
\frac{|Y|}{|\Omega(Y)|}
\sum_{K\in\Omega(Y)}
\sum_{L\in\mathcal{L}_K}
\frac{1}{[L]}
z^{[L]}.
\]
Furthermore, since all~$K\in\Omega(Y)$ satisfy~$[K]=|Y|$, it follows that
\begin{eqnarray*}
\zeta_\alpha(z) 
& = &
\prod_{Y\in\mathcal{Y}}
\exp
\frac{1}{|\Omega(Y)|}
\sum_{K\in\Omega(Y)}
\sum_{L\in\mathcal{L}_K}
\frac{[K]}{[L]}
z^{[L]}
\\
& = &
\prod_{Y\in\mathcal{Y}}
\exp
\sum_{K\in\Omega(Y)}
\frac{1}{|\Omega(Y)|}
\sum_{L\in\mathcal{L}_K}
\frac{1}{[K:L]}
(z^{[K]})^{[K:L]}
\\
& = &
\prod_{Y\in\mathcal{Y}}
\sum_{K\in\Omega(Y)}
\frac{1}{|\Omega(Y)|}
\zeta_{\tau(Y)}(z^{|Y|}).
\end{eqnarray*}
The required result follows by factoring~$\zeta_{\tau(Y)}(z^{|Y|})$ out from the sum. 
\end{proof}

By now employing the Dirichlet series~$\Delta_L$, defined for each subgroup $L\in\mathcal{L}_G$ by~(\ref{exponent_dirichlet_series}), we are able to prove  Theorem~\ref{explicit_product_theorem}.

\begin{proof}[Proof of Theorem~\ref{explicit_product_theorem}]
First consider the case where~$\tau$ is the trivial action of a group $L\in\mathcal{L}_G$ on a singleton~$Y$.
By the definition of $\Delta_L$, we have $\zeta_L(z)=\zeta(z)\Delta_L(z-1)$ and Dirichlet convolution gives
\begin{equation}\label{dreamed_of_incredible_heights}
a_L(n)=\sum_{d\mid n}d\cdot b_L(d).
\end{equation}
Furthermore, since $\nfix_\tau(L)=1$ for all $L\in\mathcal{L}_G$, using elementary operations on formal power series, we obtain the functional equation,
\[
z\frac{\zeta'_\tau(z)}{\zeta_\tau(z)}
=
\sum_{n\geqslant 1} a_L(n) z^n.
\]
Hence,~(\ref{dreamed_of_incredible_heights}) together with a standard application of Lambert series gives
\[
z\frac{\zeta'_\tau(z)}{\zeta_\tau(z)}
=
\sum_{n\geqslant 1} \sum_{d\mid n}d\cdot b_L(d) z^n
=
\sum_{n\geqslant 1} n\cdot b_L(n) \frac{z^n}{1-z^n}.
\]
Dividing by~$z$ and solving the resulting differential equation, we obtain 
\[
\log\zeta_\tau(z)=-\sum_{n\geqslant 1}b_L(n)\log(1-z^n).
\]
Thus,
\begin{equation}\label{you_dress_yourself_up}
\zeta_\tau(z)=\prod_{n\geqslant 1}(1-z^n)^{-b_L(n)}.
\end{equation}

In the general case, where~$\alpha$ is a $G$-action on a set~$X$ such that~$\nfix_\alpha(L)<\infty$, for all~$L\in\mathcal{L}_G$, the required product formula~(\ref{now_is_the_time}) now follows from~(\ref{you_dress_yourself_up}) and Lemma~\ref{basic_dynamical_zeta_product_lemma}.

To see that the formal power series for~$\zeta_\alpha$ has integer coefficients when~$\Delta_{L}$ has integer coefficients for all~$L\in\mathcal{L}_G$, we simply need to expand the product~(\ref{now_is_the_time}) in terms of geometric series, as Lemma~\ref{moebius_lemma} shows there can be only finitely many orbits of a given cardinality.
\end{proof}

\begin{remark}
For an action generated
by a single invertible transformation, the acting group is $G=\mathbb{Z}$ and all subgroups in $\mathcal{L}_G$ are isomorphic to $\mathbb{Z}$. So, for all finite orbits~$Y$, $\Delta_{L(Y)}(z)=\zeta_{\mathbb{Z}}(z+1)/\zeta(z+1)=1$, and so $b_{L(Y)}(1)=1$ and $b_{L(Y)}(n)=0$, for all $n\geqslant 2$. Consequently, Theorem~\ref{explicit_product_theorem} gives the familiar product formula~(\ref{classical_case_product_formula}) in the case $G=\mathbb{Z}$.
\end{remark}

Examples of the following type have been considered in~\cite{MR1978372} and~\cite{ETS:9315807}. Such examples show that integrality of the coefficients in the formal power series expansion for~$\zeta_\alpha$ cannot be expected in general. We also see how the principal product formula given in~\cite{MR1978372} may be deduced from Lemma~\ref{basic_dynamical_zeta_product_lemma}, by setting $G=\mathbf{D}_\infty$.

\begin{example}
The subgroups of the infinite dihedral group~$\mathbf{D}_\infty$ are described in Example~\ref{block_system_example}. This description gives~$\zeta_{\mathbf{D}_\infty}(z)=2^{-z}\zeta(z)+\zeta(z-1)$ and $\Delta_{\mathbf{D}_\infty}(z)=2^{-1-z}+\zeta(z)/\zeta(z+1)$, which does not have integer coefficients. In principle, the rational exponents in Theorem~\ref{explicit_product_theorem} may be calculated using $\Delta_{\mathbf{D}_\infty}$. However, for this example we appeal to Lemma~\ref{basic_dynamical_zeta_product_lemma} directly. 

Suppose $\alpha$ is an action of $\mathbf{D}_\infty$ on a set $X$ such that that~$\nfix_\alpha(L)<\infty$ for all~$L\in\mathcal{L}_{\mathbf{D}_\infty}$. There are just two isomorphism classes in~$\mathcal{L}_{\mathbf{D}_\infty}$ corresponding to~$\mathbf{D}_\infty$ and~$\mathbb{Z}$ and these partition the orbits of~$\alpha$ into two sets~$\mathcal{Y}_{\mathbf{D}_\infty}$ and~$\mathcal{Y}_{\mathbb{Z}}$ respectively, depending on the isomorphism class of the stabilizers of the elements in the orbit. 
Therefore, upon calculating the individual dynamical zeta functions for the trivial actions of~$\mathbf{D}_\infty$ and~$\mathbb{Z}$ on a point, Lemma~\ref{basic_dynamical_zeta_product_lemma}
gives
\[
\zeta_\alpha(z)=
\left(
\prod_{Y\in\mathcal{Y}_{\mathbf{D}_\infty}}
\frac{1}{\sqrt{1-z^{2|Y|}}}\exp\frac{z^{|Y|}}{1-z^{|Y|}}
\right)
\left(
\prod_{Y\in\mathcal{Y}_{\mathbb{Z}}}\frac{1}{1-z^{|Y|}}
\right),
\]
precisely as in~\cite{MR1978372}.

Notice that the trivial action~$\tau(\mathbf{D}_\infty)$ of the infinite dihedral group on a point has dynamical zeta function  
\begin{equation}\label{zeta_function_of_infinite_dihedral_on_a_point}
\zeta_{\tau(\mathbf{D}_\infty)}(z)=
\exp\sum_{n\geqslant 1}\frac{a_{\mathbf{D}_\infty}(n)}{n}z^n
=
\frac{1}{\sqrt{1-z^2}}\exp\frac{z}{1-z},
\end{equation}
which can be written explicitly as the power series
\[
1 + z + 2z^2 + \frac{8}{3}z^3 + \frac{25}{6}z^4 + \frac{169}{30}z^5 + \frac{361}{45}z^6 + \frac{3364}{315}z^7 +\dots
\]
\end{example}

\begin{remarks}
\noindent
\begin{enumerate}
\item
More generally, orbits may be partitioned similarly to the example above, giving an alternative formulation of the product formula in Lemma~\ref{basic_dynamical_zeta_product_lemma}. That is, if $\mathcal{S}(\alpha)\subset\mathcal{L}_G$ is a set of representatives for the isomorphism classes of stabilizers in $\mathcal{L}_G$ and if for each $L\in\mathcal{S}(\alpha)$,~$\mathcal{Y}_L$ denotes the set of orbits with stabilizers isomorphic to~$L$, then
\begin{equation}\label{alternative_product_formula}
\zeta_\alpha(z)=\prod_{L\in\mathcal{S}(\alpha), Y\in\mathcal{Y}_L}\zeta_{\tau(L)}(z^{|Y|}),
\end{equation}
where $\zeta_{\tau(L)}$ denotes the trivial action of $L$ on a point.
\item
The problem of determining if the Dirichlet series~$\Delta_L$ has integer coefficients for all~$L\in\mathcal{L}_G$ appears to be highly non-trivial in general. On the surface, it requires explicit knowledge of the isomorphism classes in~$\mathcal{L}_G$, although it seems to be a property enjoyed by particular classes of groups. For instance, as a result of the explicit formulae available for~$\zeta_G$, for each nilpotent group~$G$ considered in Section~\ref{full_shift_section},~$\Delta_G$ is seen to have integer coefficients. If this is the case for all nilpotent groups, then Theorem~\ref{explicit_product_theorem} would show that the formal power series for~$\zeta_\alpha$ has integer coefficients for any action~$\alpha$ of a nilpotent group, as every subgroup of a nilpotent group is nilpotent. This issue is discussed further in Section~\ref{concluding_remarks}. Note that the subsequent examples involving planar symmetry groups show that~$\Delta_G$ may have integer coefficients for non-nilpotent~$G$. 
\end{enumerate}
\end{remarks}

\section{Full shifts and related examples}\label{full_shift_section}
Let~$G$ be a finitely generated group, let~$\mathcal{A}$ be a finite alphabet of cardinality $A$, and let~$X=\mathcal{A}^G$. The \emph{full shift} action~$\sigma(G)$ of~$G$ on~$X$ is defined as follows. For each~$x=(x_g)$ and each~$h\in G$, let~$(hx)_g=x_{hg}$. The full shift $\sigma(G)$ on $\mathcal{A}^G$ is an expansive action with respect to any metric compatible with the product topology on $\mathcal{A}^G$.  Furthermore, for each~$L\in\mathcal{L}_G$, we have
\[
\nfix_{\sigma(G)}(L)=A^{[L]}. 
\]
To see this, choose a set of right coset representatives~$K$ for the subgroup~$L$. Each~$g\in G$ can be written uniquely in the form~$\ell k$, where~$\ell\in L$ and~$k\in K$, and~$x\in\fix_{\sigma(G)}(L)$ if and only if for all~$h,\ell\in L$ and all~$k\in K$,~$x_{h\ell k}=x_{\ell k}$, that is, if and only if for all~$h\in L$ and all~$k\in K$,~$x_{hk}=x_{k}$. So, the points of~$\fix_{\sigma(G)}(L)$ are determined by the~$A^{|K|}=A^{[L]}$ free choices for the coordinates~$x_{k}$,~$k\in K$. Thus, if~$\tau(G)$ denotes the trivial action of~$G$ on a point, then using the same method as in the proof of Theorem~\ref{explicit_product_theorem}, we obtain the following refinement for full shifts.

\begin{theorem}\label{full_shifts_theorem}
Let~$G$ be a finitely generated group and let~$\sigma(G)$ be the full shift of~$G$ on~$\mathcal{A}^G$, where~$\mathcal{A}$ is a finite alphabet. Then 
\[
\zeta_{\sigma(G)}(z)
=
\zeta_{\tau(G)}(Az)
=
\prod_{n\geqslant 1}(1-A^n z^n)^{-b_G(n)},
\]
where $A=|\mathcal{A}|$ and the exponents~$b_G(n)\in\mathbb{Q}$, $n\geqslant 1$, are the coefficients of the Dirichlet series~$\Delta_G(z)$ given by~(\ref{exponent_dirichlet_series}). Hence, the formal power series for~$\zeta_{\sigma(G)}(z)$ has integer coefficients provided that the Dirichlet series~$\Delta_G$ has integer coefficients.
\end{theorem}

In conjunction with Proposition~\ref{natural_boundary_proposition}, we thus obtain the proof of Theorem~\ref{full_shift_natural_boundary_theorem}.

\begin{proof}[Proof of Theorem~\ref{full_shift_natural_boundary_theorem}]
Assuming that~$\Delta_G$ has integer coefficients, the product formula of Theorem~\ref{full_shifts_theorem} shows that the formal power series for~$\zeta_{\sigma(G)}$ has integer coefficients, as a consequence of the formal power series for~$\zeta_{\tau(G)}$ having integer coefficients. In particular, Proposition~\ref{natural_boundary_proposition} implies that~$\zeta_{\tau(G)}$ has the unit circle as a natural boundary, since $\ugr(\tau(G))=0$. Moreover, since $\zeta_{\sigma(G)}(z)=\zeta_{\tau(G)}(Az)$, it follows that the circle $|z|=A^{-1}$ is a natural boundary for~$\zeta_{\sigma(G)}$.
\end{proof}

Table~\ref{examples_with_integer_coefficients} gives some examples where the  formulae available for~$\zeta_G$ (see \cite[Ch.~15]{MR1978431} and \cite[Sec.~4]{MR1710163}) can be used in combination with Theorem~\ref{full_shifts_theorem} to find explicit formulae for~$\zeta_{\sigma(G)}$. These examples are of particular interest because~$\Delta_G$ has integer coefficients in all cases, and hence so too does the formal power series for~$\zeta_{\sigma(G)}$. Here,~$\mathbf{H}$ denotes the discrete Heisenberg group and we have used the Hermann--Mauguin notation~$\mathbf{pg}$,~$\mathbf{pm}$ and~$\mathbf{cm}$ for the examples of planar symmetry groups, which are finite extensions of~$\mathbb{Z}^2$. Theorem~\ref{full_shift_natural_boundary_theorem} applies to all the examples in Table~\ref{examples_with_integer_coefficients}, with the obvious exception of $G=\mathbb{Z}$.

\begin{center}
\begin{table}[h]
\caption{Examples of groups for which~$\Delta_G$ and the formal power series for~$\zeta_{\sigma(G)}(z)$ have integer coefficients.\label{examples_with_integer_coefficients}}
\begin{tabular}{l|l|l}
$G$ & Presentation &~$\Delta_G(z)$\\
\hline
$\mathbb{Z}$ &~$\langle a \rangle$ & 1 \\
$\mathbb{Z}^d$, ~$d\geqslant 2$ &~$\langle a_1,\dots,a_d:[a_j,a_k]=1,1\leqslant j,k\leqslant d\rangle$ & ~$\prod_{k=0}^{d-2}\zeta(z-k)$\\
$\mathbf{pg}$ &~$\langle a,b,c: [a,b]=1, c^2=a, bc=cb^{-1}\rangle$ &~$\zeta(z)$\\
$\mathbf{pm}$ &~$\langle a,b,c: [a,b]=[a,c]=c^2=1, bc=cb^{-1}\rangle$ & ~$(1+2^{-z+1})\zeta(z)$\\
$\mathbf{cm}$ &~$\langle a,b,c: [a,b]=c^2=1, bc=cb^{-1}, ac=cab\rangle$ &~$(1+4^{-z})\zeta(z)$\\ 
$\mathbf{H}$ &~$\langle a,b,c: [a,b]=c, [a,c]=[b,c]=1\rangle$ &~$\dfrac{\zeta(z)\zeta(2z)\zeta(2z-1)}{\zeta(3z)}$ \\
\end{tabular}
\end{table}
\end{center}

We may calculate dynamical zeta functions directly using the final column in Table~\ref{examples_with_integer_coefficients} in combination with Theorem~\ref{full_shifts_theorem}. For a full shift of~$\mathbb{Z}^2$, Lind~\cite{MR1411232} observed that the dynamical zeta function has the form $P(Az)$, where 
\[
P(z)=\prod_{n\geqslant 1}(1-z^{n})^{-1}
\]
is the additive partition function. It is also evident that the full shift of~$\mathbf{pg}$ has the same dynamical zeta function, since
\[
\zeta_\mathbf{pg}(z)=\zeta_{\mathbb{Z}^2}(z)=\zeta(z)\zeta(z-1)
\]
by~\cite[Th.~1.3]{MR1710163}. The formula for $\Delta_\mathbf{pm}$ in Table~\ref{examples_with_integer_coefficients} gives  
\begin{eqnarray*}
\zeta_{\tau(\mathbf{pm})}(z)
& = &
\prod_{n\geqslant 1}(1-z^n)^{-1}(1-z^{2n})^{-2}\\
& = &
P(z)P(z^2)^2,
\end{eqnarray*}
and a similar calculation shows~$\zeta_{\tau(\mathbf{cm})}(z)=P(z)P(z^4)$. Finally, for the discrete Heisenberg group, by using $\Delta_\mathbf{H}$, we find
\[
\zeta_{\tau(\mathbf{H})}(z)=
\prod_{\ell,m,n\geqslant 1}P(z^{\ell^2 m^2n^3})^{m\mu(n)}.
\]
In all cases, $\zeta_{\sigma(\mathbf{pm})}(z)=\zeta_{\tau(\mathbf{pm})}(Az)$.

There is a rich source of examples closely related to full shifts available for groups given in the form of a semi-direct product $G=K\rtimes H$, where $H\leqslant G$ and $K\lhd G$. Since each element of $G$ can be expressed uniquely as a pair~$kh$, with $k\in K$ and $h\in H$, and since $K$ is a normal subgroup, we may define an action of $G$ on the shift space $X=\mathcal{A}^K$ via $(khx)_j=x_{khjh^{-1}}$, where~$x=(x_j)\in X$. Using this construction, the following example illustrates how the subgroup conjugacy class structure in $\mathcal{L}_G$ features significantly in examples only slightly more complicated than full shifts.

\begin{example}\label{projected_full_shift_of_pm}
Consider the shift space~$X=\mathcal{A}^{\mathbb{Z}^2}$ and the action~$\alpha$ of the group~$\mathbf{pm}\cong\mathbb{Z}^2\rtimes\mathbf{C}_2$ on~$X$ generated by
\[
(ax)_{i,j}=x_{i+1,j}\;\;\;
(bx)_{i,j}=x_{i,j+1}\;\;\;
(cx)_{i,j}=x_{i,-j}\;\;\;
\] 
where~$x=(x_{i,j})\in X$. 

As shown in \cite[Sec.~5]{MR1710163}, there are four isomorphism classes of subgroups in~$\mathcal{L}_\mathbf{pm}$ corresponding to~$\mathbf{pm}$,~$\mathbf{pg}$,~$\mathbf{cm}$ and~$\mathbf{p1}\cong\mathbb{Z}^2$. Furthermore, the conjugacy classes of subgroups can be split into families parameterized by these isomorphism classes and pairs of positive integers~$(k,m)$, as summarized in Table~\ref{subgroup_conjugacy_classes_for_pm}. The final family appearing in Table~\ref{subgroup_conjugacy_classes_for_pm} corresponds to the isomorphism class~$\mathbf{p1}$ and requires an additional parameter $0\leqslant j\leqslant\lfloor k/2\rfloor$. The function $\theta_k:\mathbb{Z}\rightarrow \{0,1\}$ appearing for these conjugacy classes is defined by
\[
\theta_k(j)
=
\left\{
\begin{array}{ll}
1 & \mbox{if }j=0,\\
1 & \mbox{if }2\mid k\mbox{ and }j=k/2,\\
0 &\mbox{otherwise},
\end{array}
\right.
\]
and is used to give the appropriate normalizer. For a given pair of positive integers $(k,m)$, there are precisely~$k$ subgroups in $\bigcup_{j=0}^{\lfloor k/2\rfloor}\mathcal{L}_\mathbf{p1}^j(k,m)$, as $\mathcal{L}_\mathbf{p1}^j(k,m)$ is comprised of a conjugate pair whenever $\theta_k(j)=0$, and is a singleton otherwise. 

\begin{center}
\begin{table}[h]
\caption{Subgroup conjugacy classes in~$\mathcal{L}_\mathbf{pm}$ for Example~\ref{projected_full_shift_of_pm}.\label{subgroup_conjugacy_classes_for_pm}}
\begin{tabular}{l|l|l|l|l}
Family & Representative ($L$) &~$[L]$ &~$N_\mathbf{pm}(L)$ &~$[N_\mathbf{pm}(L)]$\\
\hline
$\mathcal{L}_\mathbf{pm}^{1}(k,m)$
&
$\langle a^k,b^{2m},c\rangle$
&
$2km$
&
$\langle a,b^{m},c\rangle$
&
$m$\\
$\mathcal{L}_\mathbf{pm}^{2}(k,m)$
&
$\langle a^k,b^{2m},cb\rangle$
&
$2km$
&
$\langle a,b^{m},cb\rangle$
&
$m$\\
$\mathcal{L}_\mathbf{pm}^{3}(k,m)$
&
$\langle a^k,b^{2m-1},c\rangle$
&
$k(2m-1)$
&
$\langle a,b^{2m-1},c\rangle$
&
$2m-1$\\
$\mathcal{L}_\mathbf{pg}^{1}(k,m)$
&
$\langle a^{2k},b^{2m},ca^k\rangle$
&
$4km$
&
$\langle a,b^{m},c\rangle$
&
$m$\\
$\mathcal{L}_\mathbf{pg}^{2}(k,m)$
&
$\langle a^{2k},b^{2m},ca^kb\rangle$
&
$4km$
&
$\langle a,b^{m},cb\rangle$
&
$m$\\
$\mathcal{L}_\mathbf{pg}^{3}(k,m)$
&
$\langle a^{2k},b^{2m-1},ca^k\rangle$
&
$2k(2m-1)$
&  
$\langle a,b^{2m-1},c\rangle$
&
$2m-1$\\
$\mathcal{L}_\mathbf{cm}^1(k,m)$
&
$\langle a^{k}b^m,b^{2m},c\rangle$
&
$2km$
&
$\langle a,b^m,c\rangle$
&
$m$\\
$\mathcal{L}_\mathbf{cm}^2(k,m)$
&
$\langle a^{k}b^m,b^{2m},cb\rangle$
&
$2km$
&
$\langle a,b^m,cb\rangle$
&
$m$\\
$\mathcal{L}_\mathbf{p1}^j(k,m)$
&
$\langle a^k,a^j b^m\rangle$
&
$2km$
&
$\langle a,b,c^{\theta_k(j)}\rangle$
&
$2-\theta_k(j)$
\\
\end{tabular}
\end{table}
\end{center}

For each subgroup in a given conjugacy class, the number of points fixed by the subgroup is the same. This quantity has been calculated for each representative subgroup~$L$ and is shown in Table~\ref{fixed_points_for_conjugacy_classes}, along with the associated contributions to~$\zeta_\alpha(z)$. 
For example, the contribution from $\mathcal{L}_\mathbf{pm}^1(k,m)$ is 
\begin{eqnarray*}
\exp
\sum_{k,m\geqslant 1}\frac{m\cdot A^{k(m+1)}}{2km}z^{2km}
& = &
\exp\sum_{m\geqslant 1}\frac{1}{2}\log(1-A^{m+1}z^{2m})^{-1}\\
& = &
Q(A;Az^2)^{1/2}.
\end{eqnarray*}
where $A=|\mathcal{A}|$ and $Q(w;z)=\prod_{n\geqslant 1}(1-wz^n)^{-1}$ (so,~$Q$ is the reciprocal of an infinite shifted factorial and $Q(1;z)=P(z)$). Notice that the combined contribution to~$\zeta_\alpha(z)$ from the conjugacy classes corresponding to the isomorphism class of~$\mathbf{pg}$ is~$Q(1;Az^2)^{1/2}$.
Hence, the product of the contributions appearing in the final column of Table~\ref{fixed_points_for_conjugacy_classes} gives
\[
\zeta_\alpha(z) 
=
Q(z^{-1};Az^2) 
Q(A;Az^2)
Q(1;Az^2)^{2}.
\]  

\begin{center}
\begin{table}[h]
\caption{Contributions to~$\zeta_\alpha(z)$\label{fixed_points_for_conjugacy_classes} for Example~\ref{projected_full_shift_of_pm}}
\begin{tabular}{l|l|l}
Class ($\mathcal{L}$) &~$\nfix(L)$,~$L\in\mathcal{L}$ &~$\exp\left(\sum_{L\in\mathcal{L}}\frac{\nfix(L)}{[L]}z^{[L]}\right)$\\
\hline
$\mathcal{L}_\mathbf{pm}^{1}(k,m)$
&
$A^{k(m+1)}$
&
$Q(A;Az^2)^{1/2}$
\\
$\mathcal{L}_\mathbf{pm}^{2}(k,m)$
&  
$A^{km}$
&
$Q(1;Az^2)^{1/2}$
\\
$\mathcal{L}_\mathbf{pm}^{3}(k,m)$
&
$A^{km}$
&
$Q(z^{-1};Az^2)$
\\
$\mathcal{L}_\mathbf{pg}^{1}(k,m)$
&
$A^{2km}$
&
$Q(1;A^2z^4)^{1/4}$
\\
$\mathcal{L}_\mathbf{pg}^{2}(k,m)$
&
$A^{2km}$
&
$Q(1;A^2z^4)^{1/4}$
\\
$\mathcal{L}_\mathbf{pg}^{3}(k,m)$
&
$A^{k(2m-1)}$
&
$Q(A^{-1}z^{-2};A^2z^4)^{1/2}$
\\  
$\mathcal{L}_\mathbf{cm}^1(k,m)$
&
$A^{k(m+1)}$
&
$Q(A;Az^2)^{1/2}$
\\
$\mathcal{L}_\mathbf{cm}^2(k,m)$  
&
$A^{km}$
&
$Q(1;Az^2)^{1/2}$
\\
$\bigcup_{j=0}^{\lfloor k/2\rfloor}\mathcal{L}_\mathbf{p1}^j(k,m)$
&
$A^{km}$
&
$Q(1;Az^2)^{1/2}$
\\
\end{tabular}
\end{table}
\end{center}
\end{example}

It is worth noting that as a result of applying the formulae for $\zeta_G$ given in~\cite{MR1710163},  $\zeta_{\sigma(G)}$ appears not to have integer coefficients for all~$13$ remaining planar symmetry groups. For example, for  
\[
\mathbf{p2}=\langle a,b,c: [a,b]=c^2=1,ac=ca^{-1},bc=cb^{-1}\rangle,
\] 
we have $\Delta_{\mathbf{p2}}(z)=\zeta(z)\zeta(z-1)/\zeta(z+1)+2^{-1-z}\zeta(z)$, from which we deduce 
\[
\zeta_{\tau(\mathbf{p2})}(z)=
P(z^2)^{-1/2}\prod_{m,n\geqslant 1}P(z^{m n})^{m\mu(n)/n}.
\]
Hence, $\zeta_{\sigma(\mathbf{p2})}(z)=\zeta_{\tau(\mathbf{p2})}(Az)$ does not have integer coefficients. 

For any action~$\alpha$ of the group~$\mathbf{pm}$ on a set~$X$ with~$|\nfix_\alpha(L)|<\infty$ for all~$L\in\mathcal{L}_G$, the formulae in Table~\ref{examples_with_integer_coefficients} may also be applied in conjunction with~(\ref{alternative_product_formula}) to give a product formula for $\zeta_\alpha$ in terms of orbits. In particular, 
\[
\zeta_\alpha(z)
=
\prod_{Y\in \mathcal{Y}}
P(z^{|Y|})
\prod_{Y\in \mathcal{Y}_\mathbf{cm}}
P(z^{4|Y|})
\prod_{Y\in \mathcal{Y}_\mathbf{pm}}
P(z^{2|Y|})^2,
\]
where~$\mathcal{Y}$ denotes the set of all orbits in~$X$ and~$\mathcal{Y}_\mathbf{pm}, \mathcal{Y}_\mathbf{cm}\subset\mathcal{Y}$ denote the sets of orbits with associated stabilizers that belong to the isomorphism classes corresponding to~$\mathbf{pm}$ and~$\mathbf{cm}$ respectively. 

To contrast with the examples considered thus far, we now consider full shifts and related actions of virtually cyclic groups. It is well known that for a $\mathbb{Z}$-action~$\alpha$ generated by a hyperbolic toral automorphism or by a shift of finite type, $\zeta_\alpha$ is a rational function~\cite{MR0288786},~\cite{MR0271401}. For a $\mathbb{Z}$-action $\alpha$ generated by an ergodic automorphism of a compact abelian group, it is conjectured in~\cite{MR3217030} that if~$\zeta_\alpha$ fails to be rational, then this function is not analytically continuable beyond the circle of convergence, and therefore has a natural boundary. However, this dichotomy is clearly not the case more generally. For example, $\zeta_{\sigma(\mathbf{D}_\infty)}(z)=\zeta_{\tau(\mathbf{D}_\infty)}(Az)$,where $A$ is the cardinality of the alphabet and $\zeta_{\tau(\mathbf{D}_\infty)}$ is given by~(\ref{zeta_function_of_infinite_dihedral_on_a_point}). So,  $\zeta_{\sigma(\mathbf{D}_\infty)}$ is not rational but there is an analytic continuation of $\zeta_{\sigma(\mathbf{D}_\infty)}$ beyond the circle $|z|=A^{-1}$, and~$\zeta_{\sigma(\mathbf{D}_\infty)}$ is in fact holonomic.

The extent to which rational dynamical zeta functions are possible for actions of infinite virtually cyclic groups is not yet clear. Further to Example~\ref{virtually_cyclic_toral_example}, some examples where~$\zeta_{\sigma(G)}$ is rational are given in Table~\ref{examples_with_rational_zeta}.  In each case, $G$ is the direct product of~$\mathbb{Z}$ with a finite group and $\zeta_G$ may be calculated with aid of Goursat's lemma and GAP~\cite{GAP4}; this subsequently enables the calculation of~$\Delta_G$, and~$\zeta_{\sigma(G)}$ is then obtained using Theorem~\ref{full_shifts_theorem}. Here,~$\mathbf{D}_8$ is the dihedral group with~$8$ elements,~$\mathbf{UT}(3,3)$ is the group of unitriangular matrices over~$\mathbb{F}_3$, and~$p$ is any rational prime (whereby, we obtain infinitely many virtually cyclic groups of the form~$G=\mathbb{Z}\times(\mathbb{Z}/p\mathbb{Z})$ for which $\zeta_{\sigma(G)}$ is rational). 
\begin{center}
\begin{table}[h]
\caption{Examples where~$\zeta_{\sigma(G)}(z)=\zeta_{\tau}(Az)$ is rational.\label{examples_with_rational_zeta}}
\begin{tabular}{l|l|l}
$G$ & $\Delta_G(z)$ & $\zeta_{\tau}(z)$\\
\hline
$\mathbb{Z}\times(\mathbb{Z}/p\mathbb{Z})$ & 
$1+\frac{1}{p^z}$ & 
$(1-z)^{-1}(1-z^p)^{-1}$\\
$\mathbb{Z}\times\mathbf{D}_8$ &   
$1+\frac{3}{2^z}+\frac{3}{4^z}+\frac{1}{8^z}$ & 
$(1-z)^{-1}(1-z^2)^{-3}(1-z^4)^{-3}(1-z^8)^{-1}$\\
$\mathbb{Z}\times\mathbf{UT}(3,3)$ & $1+\frac{4}{3^z}+\frac{5}{9^z}+\frac{1}{27^z}$ &
$(1-z)^{-1}(1-z^3)^{-4}(1-z^9)^{-5}(1-z^{27})^{-1}$\\
\end{tabular}
\end{table}
\end{center}

We conclude by using variations on full shifts to show that bounds analogous to those obtained by Bowen~(\ref{bowens_inequalities}) should not be expected for expansive actions of virtually cyclic groups in general. 
In particular, there is no tight connection between the radius of convergence of $\zeta_\alpha$ and the topological entropy. Additionally, as well as giving a concrete instance of the the following proposition, the final example provides another rich source of actions related to full shifts for which~$\zeta_\alpha$ is rational.

\begin{proposition}\label{virtually_cyclic_proposition}
Let $G$ be an infinite virtually cyclic group. If $G$ is neither cyclic nor isomorphic to the infinite dihedral group, then there exists an expansive action $\alpha$ of $G$ by continuous automorphisms of a compact abelian group such that $\ugr(\alpha)>\entropy(\alpha)>0$.
\end{proposition}

\begin{proof}
Since $G$ is infinite and virtually cyclic, there exists a surjective homomorphism $\phi:G\rightarrow H$, where $H$ is either infinite cyclic or isomorphic to~$\mathbf{D}_\infty$ (see, for example,~\cite[Ch.~2]{MR3112976}). By assumption, $K=\ker(\phi)$ is non-trivial. Furthermore, we may define an expansive action~$\alpha$ of $G$ on the shift space~$\mathcal{A}^H$ via
\begin{equation}\label{projected_action}
(gx)_h=x_{\phi(g)h},
\end{equation}
where $\mathcal{A}$ is a finite alphabet and $x=(x_{h})\in\mathcal{A}^H$. A standard entropy calculation shows that this $G$-action has topological entropy $\entropy(\alpha)=\frac{1}{|K|}\log A$, where $A=|\mathcal{A}|$.

Now consider the sublattice $\mathcal{K}\subset\mathcal{L}_G$ comprised of finite index subgroups of $G$ containing~$K$. Via the correspondence theorem, every element $L\in\mathcal{K}$ satisfies $[G:L]=[H:\phi(L)]$. Furthermore, for all $L\in\mathcal{K}$
\[
\nfix_\alpha(L)=A^{[H:\phi(L)]}=A^{[G:L]},
\]
so $\ugr(\alpha)\geqslant\log A$.
\end{proof}

\begin{example}
Let $G=\mathbb{Z}\times(\mathbb{Z}/p\mathbb{Z})$, where $p$ is a rational prime. Since the full shift of $G$ has
\[
\zeta_{\sigma(G)}(z)=(1-Az)^{-1}(1-A^pz^p)^{-1},
\]
where $A$ is the cardinality of the alphabet, it follows that $\ugr(\sigma(G))=\log A=\entropy(\sigma(G))$. 

The description of~$\mathcal{L}_G$ given in Example~\ref{virtually_cyclic_toral_example} allows us to calculate~$\zeta_\alpha$ for the expansive action $\alpha$ of $G$ defined by~(\ref{projected_action}) with $\phi:G\rightarrow\mathbb{Z}$ being the natural coordinate projection. Notice that $\phi(L(n,k))=\langle n \rangle$ when $k=0$, and 
\[
\phi(L(n,k))=\langle \gcd(n,kn/p)\rangle=\langle n/p\rangle
\]
when $1\leqslant k\leqslant p-1$. Furthermore, for all $L\in\mathcal{L}_d$,  
$\nfix_\alpha(L)=A^{[\mathbb{Z}:\phi(L)]}$. Therefore, if $p\nmid n$,  
\[
\sum_{L\in\mathcal{L}_G(n)}\nfix_\alpha(L) = \nfix_\alpha(L(n,0)) = A^n.
\]
Whilst, if $p\mid n$,  
\begin{eqnarray*}
\sum_{L\in\mathcal{L}_G(n)}\nfix_\alpha(L)
& = &
\nfix_\alpha(L(n,0))+\nfix_\alpha(L(n))+\sum_{k=1}^{p-1}\nfix_\alpha(L(n,k))\\
& = &
A^n +
A^{n/p}+
\sum_{k=1}^{p-1}A^{n/p}\\
& = &
A^n+pA^{n/p}.
\end{eqnarray*}
Hence,
\[
\zeta_\alpha(z)=(1-Az)^{-1}(1-Az^p)^{-1}
\]
and so $\zeta_\alpha$ is once again rational. This time, however, $\ugr(\alpha)=\log A$, yet $\entropy(\alpha)=\frac{1}{p}\log A$. 
\end{example}

\section{Concluding Questions}\label{concluding_remarks} 

\noindent
\begin{enumerate}
\item It would be interesting to know if~$\Delta_G$ has integer coefficients whenever $G$ is nilpotent. If $G$ is nilpotent, as shown in~\cite[Ch.~15]{MR1978431}, the arithmetic function~$a_G$ is multiplicative and there is a product formula of the form $\zeta_G=\prod_{p}\zeta_{G,p}$, where $\zeta_{G,p}(z)=\sum_{i=0}^{\infty}a_G(p^i)p^{-iz}$, and~$p$ runs through all rational primes. This means that the question of whether or not~$\Delta_G$ has integer coefficients may be considered locally, as the Euler product for~$1/\zeta$ enables us to write $\Delta_G=\prod_{p}\Delta_{G,p}$, where 
\[
\Delta_{G,p}=(1-p^{-1-z})\zeta_{G,p}(z+1).
\]
Consequently,~$\Delta_G$ has integer coefficients if $p^i\mid a_G(p^i)-a_G(p^{i-1})$ for all~$i\geqslant 1$ and all primes $p$.
\item For which groups~$G$ is the dynamical zeta function rational for all expansive $G$-actions by automorphisms of compact abelian groups? Is it necessary for $G$ to be virtually cyclic? Note that the dynamical zeta function~$\zeta_{\sigma(\mathbf{D}_\infty)}$ shows that this condition is certainly not sufficient. 
\item For a fixed acting group~$G$, is it possible to have both rational and irrational zeta functions for expansive $G$-actions by automorphisms of infinite compact abelian groups? 
\item What dynamical information is encapsulated by the function  
\[
\zeta^{\lhd}_{\alpha}(z)=\exp\sum_{L\in\mathcal{L}^{\lhd}_G}\frac{\nfix_\alpha(L)}{[L]}z^{[L]},
\]
where $\mathcal{L}^{\lhd}_G\subset\mathcal{L}_G$ is the sublattice of normal subgroups of finite index? In light of the many available results for $\zeta^{\lhd}_G$, this also seems to be a very natural object to study. 
\end{enumerate}

\bibliographystyle{plain}

\end{document}